\documentclass{amsart}
\usepackage{graphicx}
\usepackage{amssymb}
\usepackage{amsfonts}
\usepackage{hyperref}
\usepackage{mathrsfs}
\swapnumbers
\newtheorem{theorem}{Theorem}[section]
\newtheorem{lemma}[theorem]{Lemma}
\newtheorem{corollary}[theorem]{Corollary}
\newtheorem{proposition}[theorem]{Proposition}
\theoremstyle{definition}
\newtheorem{definition}[theorem]{Definition}
\newtheorem{assumption}[theorem]{Assumption}
\newtheorem{remark}[theorem]{Remark}

\numberwithin{equation}{section}
\theoremstyle{plain}

\numberwithin{equation}{section} 
\numberwithin{figure}{section} 
\theoremstyle{plain}
\theoremstyle{plain}
\theoremstyle{remark}
\newtheorem*{acknowledgement*}{Acknowledgement}
\theoremstyle{example}

\newcommand{\cB}{{\mathcal B}}

\newcommand{\cF}{{\mathcal F}}

\newcommand{\cN}{{\mathcal N}}

\newcommand{\cP}{{\mathcal P}}

\newcommand{\te}{{\theta}}

\newcommand{\om}{{\omega}}

\newcommand{\del}{{\delta}}

\newcommand{\sig}{{\sigma}}
\newcommand{\al}{{\alpha}}

\newcommand{\la}{{\lambda}}

\newcommand{\bbE}{{\mathbb E}}

\newcommand{\bbN}{{\mathbb N}}
\newcommand{\bbP}{{\mathbb P}}
\newcommand{\bbR}{{\mathbb R}}

\newcommand{\bbZ}{{\mathbb Z}}

\begin{document}
\title[]{A few notes on the asymptotic behavior of Rademacher random multiplicative functions}   
 \vskip 0.1cm
 \author{Yeor Hafouta}
\address{
Department of Mathematics, The University of Florida}
\email{yeor.hafouta@mail.huji.ac.il}%

\thanks{ }
\dedicatory{  }
 \date{\today}

\maketitle
\markboth{Y. Hafouta}{ } 
\renewcommand{\theequation}{\arabic{section}.\arabic{equation}}
\pagenumbering{arabic}

\begin{abstract}
Let $\cP$ be the set of prime numbers, and  $X_p,\, p\in\cP$ be a sequence of independent random variables such that $\bbP(X_p=\pm 1)=1/2$.
Let $(\te_j)_{j=1}^{\infty}$ the corresponding random multiplicative functions of Radamacher type, namely, $\te_j=\prod_{p|j}X_p$ if $j$ is square free and $\te_j=0$ otherwise.  The motivation behind considering these variables comes from  the Riemann Hypothesis since $(\te_\ell)$ can be viewed as the random counterpart of the M\"obius function $\mu(\ell)$ and the RH is equivalent to $\sum_{\ell\leq n}\mu(\ell)=o(n^{1/2+\varepsilon}), \forall \varepsilon>0$. 
Denote $S_n=\sum_{\ell=1}^n\te_\ell$. Then from this point of view proving limit theorems for $S_n$ is natural problem, since $S_n$ mimics the behavior of $\sum_{\ell\leq n}\mu(\ell)$.
It is a natural guiding conjecture that $S_n/\sqrt n$ obeys the  central limit theorem (CLT). However,
 S. Chatterjee conjectured (as expressed in
\cite{[25]}) that the CLT should not hold.  Chatterjee's conjecture was proved by Harper \cite{[17]}, and by now it is a direct
consequence of a more recent breakthrough by Harper \cite{Har20}  that $\frac{S_n}{b_n}\to 0$ in $L^1$, where $b_n=(n^{1/2}(\ln(\ln(n)))^{-1/4})u_n, u_n\to\infty$. In particular $S_n/\sqrt n\to 0$.
Nevertheless, the question whether there exists a sequence $a_n=o(b_n)$ such that $S_n/a_n$ converges to some limit remains a mystery.  Note that the corresponding problem in the Steinhaus setting was recently resolved by  \cite{Gor1}.  
In this paper we make an attempt to shed some light on the convergence of $S_n/a_n$. Additionally, assuming the Riemann Hypothesis
we obtain explicit estimates on high moments of $S_n$ without restrictions on the size of the moment compared to $n$ like in \cite[Theorem 1.2]{Har19}, which is of independent interest. This is achieved by a martingale argument together with the Burkholder inequality. Using martingale techniques we will also obtain exponential concentration inequalities for $S_n$, which also improve the upper bounds in \cite[Corollary 2]{Har20}.
\end{abstract}

\section{Introduction}\label{Intro}
The asymptotic behavior of multiplicative functions is  has been studied extensively in recent years. 
One motivation is to get a better understanding of the  M\"obius function $\mu(n)$ which is supported on square free positive integers and takes  $\mu(\ell)=(-1)^{\om(\ell)}$ if $\om(\ell)$ is the number of prime factors of $\ell$.  The
generating series is $\sum_n\mu(n)/n^s$ which is equal to $1/\zeta(s)$, where $\zeta$ denotes the Riemann zeta-function. This is why studying the partial sums 
of partial sums of $\mu(\cdot)$ is important. Indeed, the above relation yields that  the Riemann Hypothesis is equivalent to the estimate
\begin{equation}\label{(1.1)}
\sum_{\ell\leq n}\mu(n)\ll n^{\frac12+\varepsilon}    
\end{equation}
for all $\varepsilon>0$. This kind of growth is similar to the almost sure growth rates of partial sums $\sum_{\ell=1}^n \te_\ell$ of independent and identically distributed (iid) zero mean  random variables $(\te_\ell)$. In fact the law of iterated logarithm guarantees the growth rate $\sqrt{n\ln(\ln(n))}$.
The above estimate is often refereed to as  the ``pseudo-randomness” of the
M\"obius function (e.g., \cite[p. 338]{IK04}). 

 A naive reason which can be used  to  understand why \eqref{(1.1)} should be true is to think about the M\"obius function as a sequence of
iid random variables $\te_\ell, \ell\in\text{SF}$ taking the values $\pm 1$ with
equal probability, where $\text{SF}$ is the set of square free positive integers.
If this were the case,  the partial sums $\sum_{\ell=1}^n\mu(\ell)$
would behave like a random walk with
mean zero and variance equal to number of squarefree integers up to $n$ (which is of order $n/\zeta(2)=6n/\pi^2$); in particular, we
would expect that \eqref{(1.1)}
should hold.

In \cite{Le31}, L\'evy pointed out a big issue in the above heuristics: the sequence of random variables $(\te_\ell), \ell\in\text{SF}$ do not have a multiplicative structure like in $(\mu(n))$. To overcome this, Wintner \cite{Win44} introduced the  random Rademacher
multiplicative functions. These random variables are defined by $\te_\ell=\prod_{p|\ell}X_p$ if $\ell\in\text{SF}$ and $\te_\ell=0$ otherwise. Here $(X_p)$ is an iid sequence such that $\bbP(X_p=\pm1)=1/2$. 
Note that $(\te_\ell)$ is a better random version of  $(\mu(\ell))$, since $(\mu(\ell))$ is one realization of $(\te_\ell)$.

In particular  Wintner's showed that, for all $\varepsilon>0$, both $\sum_{\ell\leq n}\te_\ell=O(n^{1/2+\varepsilon})$ and 
$\sum_{\ell\leq n}\te_\ell\not=O(n^{1/2-\varepsilon})$. This is considered as evidence for the validity 
the Riemann Hypothesis.
Random multiplicative functions have been studied extensively in recent years, most notably by of Harper (e.g., \cite{Har13, HNR15, Har19, Har20, Har21, Har23}).  We refer the readers to \cite[Section 2]{Har13} and the introduction
of \cite{Har19,Har21} for a comprehensive overview of this subject.




A classical question which attracted a lot of attention is to understand the distribution of the the partial sums $S_n=\sum_{\ell=1}^n\te_\ell$.
It is a natural guiding conjecture that the central limit theorem (CLT) holds, namely that
$$
S_n/\sqrt n\to\cN(0,a),\,\text{ as }\,n\to\infty
$$
in distribution, where $a$ is the density of $\text{SF}$ ($a=\frac{6}{\pi^2}$) and $\cN(0,a)$ is the normal distribution with zero mean and variance $a$.
However S. Chatterjee suggested that this should
not hold. Chatterjee’s conjecture (see \cite{[25]}), was proved by Harper \cite{[17]}, using an intricate conditioning argument. It is now a direct
consequence of a more recent breakthrough work by Harper \cite{Har20} that, in
fact,
$$
\lim_{n\to\infty}\frac{S_n}{C_nb_n}=0\,\,\text{ in }L^1,\,\,\, b_n=n^{1/2}(\ln(\ln(n)))^{-1/4},\, C_n\to\infty.
$$ 
Note that  $\|n^{-1/2}S_n\|_{L^2}^2\asymp \frac{6}{\pi^2}$
and so the convergence to $0$ does not hold in $L^2$. 
If one restricts to several natural subsums, Chatterjee
and Soundararajan \cite{[6]}, Harper \cite{[17]} and Hough \cite{[25]} and Soundararajan and Xu \cite{CLT} proved central limit theorems. In \cite{Gor} Gorodetsky and Wong  established the CLT for weighted sums $\sum_{\ell=1}^nf(\ell)\te_\ell$ when $f(\ell)$ is a multiplicative function exhibiting some (averaged) decay conditions as $\ell\to\infty$, which excludes the classical case when $f(\ell)=1$. In a recent breakthrough  \cite{Gor1} Gorodetsky and Wong  resolved this problem in the completely multiplicative setting of Steinhaus, where $\te_{\ell}=\prod_{j}X_{p_j}^{e_j}, \ell=\prod_{j}p_j^{e_j}$ and $X_p$ are uniformly distributed on the unit circle. More precisely, they showed that 
$S_n/b_n$ converges in distribution. The question remains open in the Rademacher case.

The main difficulty in proving the CLT is
that the values $\te_n$ and $\te_m$ are not independent whenever $\text{gcd}(m,n)>1$. It is also not hard to show that the sequence $(\te_\ell)$ is not weakly dependent, and instead it exhibits strong long range dependence. 
For instance, if $\ell$ is an odd square free number then $\te_{2\ell}=\te_{\ell}X_2$, and more complicated  nontrivial long  range dependence can be described. 
The sequence $(\te_\ell)$ also does not have a  weakly locally dependent graph structure (see \cite[Ch.1]{HK}), which makes  Stein's method effective. 
Therefore, it seems unlikely that the corresponding sums $S_n$ can be directly treated using classical tools from probability that involve some kind of weak dependence.

In this note  we will study probabilistic properties of $S_n$. We first provide upper bounds on high order moments of $S_n$ and exponential concentration inequalities, assuming the validity of RH.
Compared with  \cite[Theorem 1.2]{Har19}, our results concerning moments include estimates on $\|S_n\|_{L^q}$ for every $n$ and $q$ with explicit constants that depend on $q$ (in \cite{Har19} such estimates where obtained when $q=O(\frac{\ln(n)}{\ln(\ln(n))})$). Let us note that in that range the estimates of Harper are better. This is done by applying the Burkholder inequality with the martingale decomposition of $S_n$. Without the RH we are still able to derive moment estimates but with less explicit constant. Using these results we are able to show that for a larger family of Rademacher multiplicative functions $\te_\ell$ we have 
$$
S_n=o(n^{1/2+\varepsilon}), \text{a.s. }\, \forall\varepsilon>0.
$$
We believe that this could be a source to heuristically getting a better understanding of the structure of prime numbers since we can consider random variables $X_p$ with much larger set of values than $\{-1,1\}$. Other applications of our bounds include an upper bound for the number of subsets $A$ of $\{1,2,...,n\}$ of size $2k$ which contain only square free numbers and $\prod_{a\in A}a$ is a perfect square. Additionally, using martingale techniques we will also prove exponential concentration inequalities for $S_n$, which improve the upper bounds in \cite[Corollary 2]{Har20}. 
Finally, we will make an attempt to shed some light on the weak convergence of $S_n/a_n$ and show that when $a_n$ grows at a certain rate there could be no weak limit.

\subsection*{Polynomial times: A brief discussion}
To appreciate the peculiarity of our results and to provide a more complete picture we will briefly describe a related class of problems. 
Let $P\in\bbZ[x]$ be a nonlinear polynomial and define $Y_\ell=\te_{P(\ell)}$. Then, as opposed to the case of linear polynomials the partial sums $\sum_{\ell=1}^n Y_\ell$ obey the CLT for a wide range of nonlinear polynomials, see \cite{GAFA,TAMS Mub} and references therein. Since this paper is focused on linear times and Rademacher functions we prefer not to elaborate too much on the subject, but let us mention that the motivation behind considering polyomial time comes from the Chowla conjecture  (see again \cite{GAFA,TAMS Mub}).
This conjecture remains unsolved (see \cite{Tao} for a proof of an averaged version) but the  results in \cite{GAFA,TAMS Mub}  can be viewed as a proof of a random counterpart of the conjecture (aka the ``random Chowla conjecture"), and they provide evidence to the validity of the deterministic conjecture. We believe that the ideas in the paper are likely to yield  the CLT in the polynomial case to the more general Random Radememacher functions considered in this paper, but  we will address this problem in a different venue. The motivation behind considering more general $X_p$'s comes from discovering new patterns in the asymptotic behavior of  prime numbers. Indeed, the corresponding random multiplicative functions can model more complicated multiplicative deterministic functions, and the CLT in the random case can provide evidence to an appropriate deterministic results.

\section{Preliminaries and main result}
In this section we will consider more general independent random variables $X_p$ than the ones discussed in Section \ref{Intro}. All of our conditions will trivially hold true in the classical Rademacher case when $\bbP(X_p=\pm1)=1/2$. However, already in the iid case our conditions include a wide class of bounded symmetric random variables.

Let us begin with a description of our setup.
Let $\cP$ be the set of prime numbers and 
let $X_p, p\in\cP$ be an independent (not necessarily identically distributed) sequence of symmetric non-constant random variables such that 
\begin{equation}\label{2 3}
M:=\sup_{p\in\cP}\|X_p\|_{L^\infty}<\infty \,\text{ and }\,c_0:=\inf_{p\in\cP}\left(\text{ess-inf}(|X_p|)\right)>0.
\end{equation}
\begin{definition}
The  random multiplicative functions $\te_\ell$ of Rademacher type generated by $(X_p)$ are defined by 
$\te_\ell=\prod_{p|\ell}X_p$ when $\ell$ is square free, and otherwise  $\te_\ell=0$.
\end{definition}
The classical Rademacher setting concerns the case when $\bbP(X_p=\pm1)=1/2$, see the discussion in Section \ref{Intro}, but here we can consider more general random variables $X_p$. Set
$$
S_n=\sum_{\ell=1}^n\te_\ell.
$$

\subsection*{On the growth of $\text{Var}(S_n)$: basic observations}
Notice that $\te_\ell$ are orthogonal since $X_p$ have zero mean and two distinct square free integers cannot have the same prime factors.
Therefore,
$$
\|S_n\|_{L^2}^2=\sum_{\ell=1}^n\bbE[(\te_\ell)^2]=\sum_{\ell\in\text{SF}\cap N_n}\prod_{p|\ell}\bbE[(X_p)^2]
$$
where $\text{SF}$ is the set of all square free integers and $N_n=\{1,2,...,n\}$. 
\begin{lemma}
Let $a\leq 1\leq b$ and suppose that $a\leq \bbE[(X_p)^2]\leq b$ for all $p$.   Then    there exist absolute constants $c,C_1,C_2>0$ such that for all $n$ large enough we have
$$
C_1n^{1-\frac{c|\ln a|}{\ln(\ln(n))}}\leq \|S_n\|_{L^2}^2\leq C_2n^{1+\frac{c\ln b}{\ln(\ln(n))}}.
$$
\end{lemma}
\begin{proof}
Let $\om(\ell)$ is the prime omega function which assigns to $\ell$ the number of distinct  prime factors. Then $\te_\ell$ is a product of $\om(\ell)$ independent random variables $X_{p_i}$.
Taking into account the above formula for $\|S_n\|_{L^2}^2$ the lemma follows from the following facts. The first one is that 
 (see \cite{Mir}) the density of the square free numbers is $6/\pi^2$. The second one is that (see \cite{omega}) $\om(\ell)\leq \frac{c\ln(\ell)}{\ln(\ln \ell)}$ for some constant $c$ and all $\ell$ large enough. The third fact is that the function $g(x)=\frac{\ln x}{\ln(\ln (x))}$ is increasing on some ray $[u,\infty), u>3$.
\end{proof}
We note that the lemma shows that $\|S_n\|_{L^2}^2$ is almost linear in $n$ in the sense that it is of order $n^{1\pm\varepsilon_n}$ for some sequence $\varepsilon_n\to0$. However, we still have $\lim n^{\varepsilon_n}=\infty$, but at a very slow rate.

\begin{corollary}
If $\inf_{p}\bbE[(X_p)^2]\geq1$ then 
$V=:\liminf_{n\to\infty}\frac1n\|S_n\|_{L^2}^2>0.
$
Moreover, if $\bbE[X_p^2]=1$ for all $p$ then
 $$
V=\lim_{n\to\infty}\frac1n\|S_n\|_{L^2}^2=\frac{\pi^2}{6}.
 $$
\end{corollary}

\subsection*{Main results}

Our first result concerns the behavior of the moments of $S_n$ without explicit constants (i.e. without allowing moment growth in $n$).

\begin{theorem}\label{MomThm2}
Suppose that $M:=\sup\|X_k\|_{L^\infty}<\infty$. Let $M_n=\max(M,1)^{\frac{c\ln(n)}{\ln(\ln(n))}}$, where $c$ is any constant such that $\om(\ell)\leq \frac{c\ln(\ell)}{\ln(\ln(\ell))}$ for all $\ell$ large enough. Then for every $k\in\bbN$ there exists a constant $C_k$ such that
$$
\bbE[(S_n)^{2k}]\leq C_k\sqrt n (M_n)^{2k}(\ln(n))^{2k^2}.
$$
\end{theorem}

An immediate consequence of Theorem \ref{MomThm2} is as follows.
\begin{corollary}
For every $\varepsilon>0$ there exists a random variable $B_{\varepsilon}\in L^{[\frac{1}{2\varepsilon}]+1}$ such that for all $n$,
$$
|S_n|\leq B_{\varepsilon}n^{1/2+\varepsilon}, \,\text{ almost surely}.
$$
\end{corollary}

\begin{proof}
Let $\varepsilon>0$. Then for all $k>[1/(2\varepsilon)]+1$ have
$$
\left\|\sup_{n\in\bbN}(n^{-1/2-\varepsilon}|S_n|)\right\|_{L^{2k}}^{2k}\leq\sum_{n\geq1}n^{-k-2k\varepsilon}\bbE[|S_n|^{2k}]<\infty.
$$
\end{proof}
Note that the above corollary could be a source of heuristically getting a better understanding of the structure of the prime numbers by considering multiplicative functions with more general values than $\pm 1$.

\begin{theorem}\label{MomThm}
Suppose the Riemann hypothesis holds and that $M:=\sup\|X_k\|_{L^\infty}<\infty$. Let $M_n=\max(M,1)^{\frac{c\ln(n)}{\ln(\ln(n))}}$, where $c$ is any constant such that $\om(\ell)\leq \frac{c\ln(\ell)}{\ln(\ln(\ell))}$ for all $\ell$ large enough.
Then there is an absolute  constant $C>0$ such that 
for every $q\geq 2$ we have 
$$
\|S_n\|_{L^q}\leq C_\varepsilon\sqrt{q}M_n n^{1-\frac{1}{2(1+\sqrt{\ln n})}}.
$$
\end{theorem}
The proof of Theorem \ref{MomThm} relies on a standard representation of $S_n$ as a martingale and on Burkholder inequality. 
Compared with the results in \cite[Theorem 1.2]{Har19} we are able to provide upper bounds for all $q$ and $n$ without the restriction $q\leq \frac{c_1\ln(n)}{\ln(\ln(n))}$. Note that our estimates are effective (namely they yield upper bounds of order $o(n)$) when $q\leq n^{\frac{\epsilon}{\sqrt{\ln n}}}$ which is a much larger range of values for $q$  than in \cite[Theorem 1.2]{Har19}.
Even compared with range of $q$'s discusses in \cite[footnote 1]{Har19}  we are able to get some estimates on much larger range of $q$'s, since  $n^{\frac{\epsilon}{\sqrt{\ln n}}}$ grows faster than any power of $\ln n$. However, in the range of $q$ in the setting of Harper (and also  \cite[footnote 1]{Har19}) the actual growth rates of the moment  are better since the constants  appearing in the bounds are actually small for large $q$ and the power $1-\frac{1}{2(1+\sqrt{\ln n})}$ can be replaced by $1/2$, with a price of a logarithmic factor. Finally, note that  \cite[footnote 1]{Har19}  is based on \cite[Corollary 6.3]{13}. We find it interesting to see if in the classical Radamacher setting \cite[Corollary 6.3]{13} should yield comparable bounds.

\begin{remark}
To appreciate the estimates in Theorem \ref{MomThm} recall that 
a classical tool in number theory to obtain upper bounds on the $L^q$ norms of $S_n$ is the rough hypercontractive inequalities, see \cite[Page 2290]{Har19}. In the classical Rademacher case this is based on a result of Hal\'asz \cite{Hal}. However, the resulting upper bounds in the classical Rademacher case are 
$$
\|S_n\|_{L^{2q}}^{2q}\leq \left(\sum_{\ell\in\text{SF}\cap[1,n]}[2q-1]^{\omega(\ell)}\right)^q
$$
where $[2q-1]$ is the ceiling of $2q-1$. Using that $\omega(\ell)\leq \frac{c\ln(\ell)}{\ln(\ln(\ell))}$ we get
$$
\|S_n\|_{L^{2q}}\leq \left(\sum_{\ell\in\text{SF}\cap[1,n]}[2q-1]^{\frac{c\ln(\ell)}{\ln(\ln(\ell))}}\right)^{1/2}.
$$
Note that the last summand is of order $q^{\frac{c\ln(n)}{\ln(\ln(n))}}$.
 Thus,  the hypercontractive approach  yieldd effective (i.e. $o(n)$) estimates on the $L^{2q}$ norms only when $q$ is of logarithmic order in $n$. 
Thus by applying the Burkholder inequality we get estimates on larger domains of $q$ than the ones obtained through the hypercontractive inequalities.
\end{remark}


Our next result concerns exponential concentration inequalities.
\begin{theorem}\label{Concent}
Assume the Riemann hypothesis.
Suppose that $M:=\sup_{p}\|X_p\|_{L^\infty}<\infty$. Then there exist constants $c,c_0>0$ such that 
for every $t_0>0$ and $n\in\bbN$,
$$
\bbP(S_n\geq t_0)\leq e^{-t_0^2/u_n}
$$
where with  $M_n=\max(M,1)^{c\frac{\ln(n)}{\ln(\ln(n))}}$, 
$u_n=c_0 M_n n^{2-\frac{1}{1+\sqrt{\ln n}}}$. In particular, when $M=1$ for every $\varepsilon>0$
$$
\bbP(S_n\geq \varepsilon n)\leq e^{-c_1\epsilon^2n^{\frac{1}{1+\sqrt{\ln n}}}}.
$$
for some constant $c_1>0$.
\end{theorem}
Note that $e^{-c_1\epsilon^2n^{\frac{1}{1+\sqrt{\ln n}}}}$ decays faster than any power on $n$. 
Theorem \ref{Concent} is proven by a standard martingale decomposition (see \cite{Har13}) together with the Azuma-Hoeffding inequality.   Compared with the upper bounds in \cite[Corollary 2]{Har20} we obtain better rates of decay of the tails. For instance,  \cite[Corollary 2]{Har20}  yields 
$$
\bbP(S_n\geq \varepsilon n)\leq \frac{\min(\ln\lambda_\varepsilon(n), \sqrt{\ln(\ln n)})}{(\la_\varepsilon(n))^2}
$$
where $\lambda_\varepsilon(n)=\epsilon\sqrt n(\ln(\ln n))^{1/4}$. Note that the above right hand side is of order a bit larger than $1/n$.

Our next result concerns the possibility of limiting distribution for $S_n/a_n$. For that purpose we need the following assumption.
\begin{assumption}\label{Ass}
 Denote by $\cB_{k,n}$ the number of subsets $A$ of $\{1,2,...,n\}$ of size $k$ which contain only square free numbers and   $\prod_{a\in A}a$ is a perfect square.
For every $t\in\bbR\setminus\{0\}$ close enough to $0$  and $\varepsilon>0$ we have
$$
 \lim_{n\to\infty}e^{-\frac{c_0nt^2}{a_n^2}}\sum_{k=\frac{c_1\ln(n)}{\ln(\ln(n))}}^{n^{1/2+\varepsilon}}\frac{t^{2k}M_n^{2k}\cB_{2k,n}}{a_n^{2k}}=0
$$
where $c_1$ is the constant from \cite[Theorem 1.2]{Har19} (denoted there by $c$).
\end{assumption}
\begin{remark}
A natural conjecture in this setting is that $\cB_{2k,n}$ diverges slow enough (in both $k$ and $n$) so that Assumption \ref{Ass} will hold when $a_n=o(b_n), b_n=n^{1/2}(\ln(\ln(n)))^{-1/4}$.
\end{remark}

\begin{theorem}\label{Main}
Let $X_p, p\in\cP$ be a sequence of symmetric independent random variables satisfying \eqref{2 3} and let us take a sequence $(a_n)$ such that $\liminf_{n\to\infty}a_n>0$. 
\vskip0.1cm
(i) If $S_n/a_n$ converges in distribution then $\liminf_{n\to\infty}\frac{a_n}{l_n}>0$, where $l_n=n^{1/2}(\ln(n))^{-1/2}$.

\vskip0.1cm
(ii) 
Let $M_n=\max(1,M)^{c\frac{\ln(n)}{\ln(\ln(n))}}$. Suppose that
 $a_n=o(\|S_n\|_{L^2})$, 
 $$ a_n=o\left(n^{1/2}c_0^{-2c\frac{\ln(n)}{\ln(\ln(n))}}(\ln(\ln(\ln(n)))^{-1}\right)\,\ \text{ and }\,\, 
 a_n=o\left(n^{1/2}M_n^{-1}c_0^{-c\frac{\ln(n)}{\ln(\ln(n))}}\right),
$$
 where $c$ is any constant for which $\om(\ell)\leq \frac{c\ln(\ell)}{\ln(\ln(\ell))}$ for all $\ell$ large enough. Then under Assumption \ref{Ass}  the sequence
 $S_n/a_n$ does not converge in distribution (not even along a subsequence).
\end{theorem}
Notice that in the classical Rademacher case $c_0=M=M_n=1$ and $\|S_n\|_{L^2}\asymp\sqrt{n}$ and so the conditions on $a_n$ holds when  $a_n=o(b_n)$ where $b_n=n^{1/2}(\ln(\ln(n)))^{-1/4}$, which we recall satisfies $S_b/c_nb_n\to 0$ for all $c_n\to\infty$. Thus, in this case  the lack of a weak limit of $S_n/a_n$ boils down to verification of Assumption \ref{Ass}.

\section{Proof of Theorems \ref{Concent} and \ref{MomThm} via martingale methods}
\subsection*{Martingale representations and $L^\infty$ type estimates.}
Let us denote by $P(n)$ the maximal prime factor of $n$ and let $\cP_n$ be the set of all prime numbers in $[1,n]$.
For $p\in\cP_n$ let
$$
M_p(n)=\sum_{\ell\leq n: P(\ell)=p}\te_{\ell}.
$$
Then $(M_p(n))_{p\in\cP_n}$ is a martingale difference with respect to the natural filtration $\cF_p=\sigma\{X_q: q\leq p\}$ (see \cite{Har13}). Clearly,
$$
S_n=\sum_{p\in\cP_n}M_p(n).
$$

Our approach will be based on the following result.
\begin{lemma}\label{L3}
Let $M:=\sup_p\|X_p\|_{L^\infty}<\infty$ and set $M_n=(\max(M,1))^{\frac{c\ln(n)}{\ln(\ln(n))}}$. 
\vskip0.1cm
(i) For every $\varepsilon\in(0,1/2)$ there exists a constant $C_\varepsilon>0$ such that for all $n\in\bbN$ 
$$
\left\|\sum_{p\in \Delta_n}M_p(n)\right\|_{L^\infty}\leq C_{\varepsilon}M_n n^{1-\frac{1-2\varepsilon}{2+\varepsilon}}
$$
where $\Delta_n=\{p\in\cP: \, p<(\ln(n/p))^{2+\varepsilon}\}$.
\vskip0.1cm
(ii) Suppose that the Riemann hypothesis holds. Then there exists  a constants $C_\varepsilon$ such that 
for  all $n\in\bbN$ and $p\in\cP_n$ so that $p\geq(\ln(n/p))^{2+\varepsilon}$,
$$
\|M_p(n)\|_{L^\infty}\leq C_\varepsilon M_n\cdot (n/p)(\ln(n/p)/\ln(p))^{-\ln(n/p)/\ln(p)}.
$$ 
\end{lemma}
\begin{proof}
Let $\psi(x,y)$ be the de Bruijn’s function which counts the number of $y$-smooth positive integers up to $x$.  Recall that by \cite{[11]} and the validity of the Riemann hypothesis if $y\geq (\ln(x))^{2+\varepsilon},\varepsilon>0$ then $\psi(x,y)\leq C_\varepsilon x\rho\left(\ln(x)/\ln(y)\right)$ for some constant $C>0$, where 
 $\rho(\cdot)$ is the Dickman function.

\vskip0.1cm
(i) To prove the first part  note that the total number of summands in $M_n(q), q\leq p$ does not exceed $\psi(n,p)$. Therefore,
Thus
$$
\sum_{q\leq (\ln(n))^{2+\varepsilon}}\|M_q(n)\|_{L^\infty}\leq M_n\psi(n,(\ln(n))^{2+\varepsilon})\leq C_\varepsilon n\rho(A_{n,\varepsilon})
$$
where $A_{n,\varepsilon}=\frac{\ln(n)}{(2+\varepsilon)\ln(\ln(n))}$. Finally, recall  that 
$\rho(u)\leq Cu^{-u}$ for some constant $C>1$. Therefore,
there exists $N_\varepsilon\in\bbN$ such that for every $n\geq N_\varepsilon$ we have 
$$
\rho(A_{n,\varepsilon})\leq n^{-\frac{1-2\varepsilon}{2+\varepsilon}}.
$$


\vskip0.1cm
(ii) Suppose $p\geq (\ln(n))^{2+\varepsilon}$.
Note that number of summands in  $M_p(n)$ is bounded by $\psi(n/p,p)$. Indeed, the terms appearing in $M_p(n)$ correspond to positive square free integers $m=pk, k\leq n/p$.
 Let us take $x=n/p$ and $y=p$.  Then if $p\geq(\ln(n/p))^{2+\varepsilon}$, using that
$\rho(u)\leq Cu^{-u}$ for some constant $C>1$ we conclude that 
$$
\|M_p(n)\|_{L^\infty}\leq C_1\max_{\ell\leq n}\|\te_\ell\|_{L^\infty}(n/p)(\ln(n)/\ln(p))^{-\ln(n/p)/\ln(p)}
$$
$$
\leq (\max(M,1))^{c\frac{\ln(n)}{\ln(\ln(n))}}(n/p)(\ln(n/p)/\ln(p))^{-\ln(n/p)/\ln(p)}.
$$   
\end{proof}

\begin{lemma}\label{L5}
Suppose that the Riemann hypothesis holds.
Let $M:=\sup_p\|X_p\|_\infty<\infty$ and set $M_n=(\max(M,1))^{\frac{c\ln(n)}{\ln(\ln(n))}}$.  There the exists a constant $C>0$ such that for all $n\geq 2$,
$$
\left\|\sum_{p\in\Delta_n}M_p(n)\right\|_{L^\infty}^2+
\sum_{p\in\cP_n\setminus \Delta_n}\|M_p(n)\|_{L^\infty}^2\leq CM_n^2 n^{2-\frac{1}{1+\sqrt {\ln n}}}.
$$
\end{lemma}
\begin{proof}
Set  $\alpha=\frac{\ln(n/p)}{\ln(p)}$. Then $p=n^{\frac{1}{1+\alpha}}$ and 
 $$
(n/p)^2\alpha^{-2\alpha}=n^{\frac{2\alpha}{1+\alpha}}\alpha^{-2\alpha}.
 $$
 Set
 $$
f_n(\alpha)=n^{\frac{2\alpha}{1+\alpha}}\alpha^{-2\alpha}=e^{\frac{2\alpha}{1+\alpha}\ln(n)-2\alpha\ln(\alpha)}:=e^{g_n(\alpha)}.
 $$
 Notice that 
 $$
g_n'(\alpha)=\frac{2\ln(n)}{(1+\alpha)^2}-2(1+\ln(\alpha)).
 $$
 Thus, by dividing the set $\cP_n$ into two monotonicity intervals and using the prime number theorem we see that there are constants $C_1,C_2>0$ such that
$$
\sum_{p\in\cP_n}(n/p)^2\alpha^{-2\alpha}\leq C_1\int_{2}^{\infty}\frac{f_n(p)}{\ln(p)}dp+C_2.
$$
Let us make a change of variables $p=n^{\frac{1}{1+\alpha}}=e^{\frac{\ln(n)}{1+\alpha}}$. Then $dp=-\ln (n)\frac{n^{\frac{1}{1+\alpha}}}{(1+\alpha)^2}$ and so 
$$
\int_{2}^{\infty}\frac{f_n(p)}{\ln(p)}dp\leq n\int_{0}^\infty \frac{n^{\frac{\alpha}{\al+1}}e^{-2\al\ln(\al)}}{(1+\alpha)^2}d\alpha.
$$
Next, define 
$$
J(n)=\int_{0}^\infty\frac{n^{-\frac{1}{\al+1}}e^{-2\al\ln(\al)}}{(1+\alpha)^2}d\alpha,\,\,\beta\geq2.
$$
Then
$$
\int_{2}^{\infty}\frac{f_n(p)}{\ln(p)}dp\leq n^2J(n).
$$
To estimate $J(n)$ let us take $d=\sqrt{\ln n}$ and write 
$$
J(n)= \int_{0}^d\frac{n^{-\frac{1}{\al+1}}e^{-2\al\ln(\al)}}{(1+\alpha)^2}d\alpha+
\int_{d}^\infty\frac{n^{-\frac{1}{\al+1}}e^{-2\al\ln(\al)}}{(1+\alpha)^2}d\alpha:=J_1(n)+J_2(n).
$$
Clearly, there is a constant $C>0$ such that 
$$
J_1(n)\leq Cn^{-\frac{1}{1+d}}=Cn^{-\frac{1}{1+\sqrt{\ln n}}}.
$$
Indeed, on the domain $0\leq \alpha\leq d$ we have $n^{-\frac{1}{\al+1}}\leq n^{-\frac{1}{1+d}}$.
Moreover, since $n^{\frac{-1}{\al+1}}\leq1$ we have 
$$
J_2(n)\leq Ce^{-2d\ln d}=d^{-2d}=(\ln n)^{-\sqrt {\ln n}}\leq C'n^{-\frac{1}{1+\sqrt{\ln n}}}
$$
where the last inequality follows by taking logarithms on both sides.
\end{proof}

\subsection*{Concentration inequalities: Proof of Theorem \ref{Concent}}
Denote $M_n=C_1(\max(M,1))^{\frac{c\ln(n)}{\ln(\ln(n))}}$.
Applying the Hoeffding–Azuma inequality (see, for instance, page 33 in \cite{[29]}) and Lemma \ref{L5} we obtain that for any $\la>0$,
$$
\bbE[e^{\la S_n}]\leq e^{\la^2(\left\|\sum_{p\in\Delta_n}M_p(n)\right\|_{L^\infty}^2+\sum_{p\in\cP_n\setminus\Delta_n}\|M_p(n)\|_{L^\infty}^2)}\leq e^{\la^2 C M_n n^{2-\frac{1}{1+\sqrt{\ln n}}}}
$$
Next, by the Markov inequality for any random
variable $Z$, $t_0>0$ and $\la>0$ we have $\bbP(Z\geq t_0)\leq e^{-\la t_0}\bbE[e^{\la Z}]$. Taking $Z=S_n$ and taking $\la$ given by 
$$
2C M_n n^{2-\frac{1}{1+\sqrt{\ln n}}}\lambda=t_0
$$
we get that 
$$
\bbP(S_n\geq t_0)\leq e^{-\frac{t_0^2}{v_n}}, \,\,v_n=4C M_n n^{2-\frac{1}{1+\sqrt{\ln n}}}
$$
\qed

\subsection*{Bounds of high order moments via Burkholder inequality: proof of Theorem \ref{MomThm}}
By applying the Burkholder inequality we see that 
$$
\|S_n\|_{L^q}^2\leq Cq\left\|\sum_{p\in\Delta_n}M_p(n)\right\|_{L^\infty}^2+Cq\sum_{p\in\cP_n\setminus\Delta_n}\|M_p(n)\|_{L^\infty}^2.
$$
Now the theorem follows from Lemma \ref{L5}.
\qed
\section{Proof of Theorem \ref{Main}}
\subsection*{Key results}
The first step is to prove Theorem \ref{Main} (i). 
\begin{lemma}\label{L1}
Suppose $\inf_{p}\|X_p\|_{L^2}>0$ and $M=\sup_p\|X_p\|_{L^\infty}<\infty$. 
If there exists a sequence $(a_n)$ such that $\liminf_{n\to\infty}a_n>0$ and $S_n/a_n$ converges in distribution  then $\liminf_{n\to\infty}\frac{a_n^2}{l_n}>0$ where $l_n=\frac{n}{\ln n}$.
\end{lemma}
The idea is  simple, but let us present a full proof.
\begin{proof}
Note that if $p$ is a prime number such that $n/2<p\leq n$ then $X_p$ appears only once in the sum $S_n$. Indeed, $X_p$ can only appear in $\te_\ell, \ell\leq n$ if $p|\ell$, which implies that $\ell\geq p$. If $\ell>p$ then $\ell$ must have another prime factor which means that $\ell\geq 2p>n$. Let  $\cP_n$ be the set of all prime numbers inside $(n/2,n]$. Set $s_n=|\cP_n|$. Then by the prime number theorem there are constants $c_1,c_2>0$ such that for all $n$ large enough 
$$
c_1\frac{\ln n}{n}\leq s_n\leq c_2\frac{\ln n}{n}.
$$

Next, denote $\varphi_p(t)=\bbE[e^{itX_p}]$.
Now, since $X_p$ are independent and we see that 
$$
\bbE[e^{itS_n}]=\left(\prod_{p\in\cP_n}\varphi_p(t)\right)\bbE[e^{it(S_n-\sum_{p\in\cP_n}X_p)}].
$$
Now,  by the Lagrange form of the Taylor remainder of order $3$ of the function $\varphi_p(\cdot)$ 
around the origin we see that for all $t$ and $n$, with
$\sig_p^2=\bbE[(X_p)^2]>0$,
\begin{equation}\label{cos}
 \left|\varphi_p(t/a_n)-\left(1-\frac{\sig_p^2t^2}{2a_n^2}\right)\right|\leq\frac{|t|^3\cdot \bbE[|X_p|^3]}{6a_n^3}\leq \frac{M|t^3|\sig_p^2}{6a_n^3}
\end{equation}
where the last inequality uses  that $\bbE[|X_p|^3]\leq M\bbE[|X_p|^2]$ (by \eqref{2 3}).
Therefore, if $|t|$ is small enough to ensure that $\frac{M|t^3|}{6a_n^3}\leq \frac{t^2}{4a_n^2}$,
then 
$$
|\varphi_p(t/a_n)|\leq 1-\frac{t^2\sig_p^2}{4a_n^2}\leq e^{-\frac{t^2\sig_p^2}{4a_n^2}}
$$
where  the last inequality uses that $1-x\leq e^{-x}$ for all $x\geq 0$.
It follows that for all $n$ large enough and every real nonzero $t$,
$$
|\bbE[e^{itS_n/a_n}]|\leq e^{-\frac{t^2}{4a_n^2}\sum_{p\in\cP_n}\sig_p^2}=e^{-\frac{t^2s_n}{4a_n^2}},\, s_n=\sum_{p\in\cP_n}\sig_p^2\asymp \frac{n}{\ln n}.
$$
 Then 
$$
|\bbE[e^{itS_n/a_n}]|\leq e^{-\frac{t^2s_n}{4a_n^2}}.
$$

 Let $Z$ be the weak limit of $S_n/a_n$. Then by Levi's continuity theorem for every $t\in\bbR$,  
$$
\lim_{n\to\infty}|\bbE[e^{itS_n/a_n}]|=|\bbE[e^{itZ}]|.
$$
Using that $t\to\bbE[e^{itZ}]$ is continuous, for all $t>0$ small enough we have $|\bbE[e^{itZ}]-1|<1/2$ and so for all $n$ large enough,
$$
\frac12<|\bbE[e^{itS_n/a_n}]|\leq e^{-t^2\frac{s_n}{4a_n^2}}.
$$
Let us take some sufficiently small nonzero point $t_0$. Then for all $n$ large enough we see that 
$$
\frac{t_0^2 s_n}{4a_n^2}\leq \ln 2,
$$
namely $a_n^2\geq \frac{t_0^2s_n}{4\ln 2}$. Therefore $\liminf_{n\to\infty}\frac{a_n}{\sqrt{s_n}}>0$.
\end{proof}

\begin{proposition}\label{L2}
Let the conditions of Theorem \ref{Main} be in force. 
Let $M=\sup_p\|X_p\|_{L^\infty}<\infty$. Let $l_n=\frac{n}{\ln n}$ and suppose $\liminf_{n\to\infty}\frac{a_n^2}{l_n}>0$. Then
there exists a constant $\del_0>0$ such that 
 for all $t\in(-\delta_0,\delta_0)\setminus\{0\}$ we have
 \begin{equation}\label{UP}
\lim_{n\to\infty}\bbE[\cos(tS_n/a_n)]=0.    
\end{equation}   
\end{proposition}
\begin{proof}
 Recall that for every absolutely convergent series $A=\sum_{k=1}^\infty a_k$ we have 
\begin{equation}\label{Coss}
 \cos(A)=\sum_{k=0}^{\infty}(-1)^k\sum_{A\subset \bbN,\, |A|=2k}\,\,\left(\prod_{\ell\in A}\sin(a_\ell)\prod_{\ell\not\in A}\cos(a_\ell)\right).   
\end{equation}
Thus, when applying the above identity with $a_\ell=t\te_\ell$ for $\ell\leq n$ and $a_\ell=0$ otherwise, and using that $\sin(0)=0$ and  $\cos(0)=1$ we see that with $N_n=\{1,2,...,n\}$ we have
\begin{equation}\label{Coss1}
\bbE[\cos(tS_n)]=\bbE\left[\prod_{\ell=1}^n\cos(t\te_\ell)\right]+\sum_{k=1}^{[n/2]}(-1)^k\sum_{A\subset N_n: |A|=2k}\bbE\left[\prod_{\ell\in A}\sin(t\te _\ell)\prod_{\ell\in N_n\setminus A}\cos(t\te _\ell)\right]
\end{equation} 
Next, by expanding $\cos(x)$ around the origin and
the Lagrange form of the error term in the Taylor expansion of order $2$ we see that for all $t,\ell$ and $n$,
$$
\left|\cos(t\te_\ell/a_n)-\left(1-\frac{t^2\te_\ell^2}{2a_n^2}\right)\right|\leq \frac{|t|^3\cdot|\te_\ell|^3}{6a_n^3}.
$$
Using also that $\liminf_{n\to\infty}\frac{a_n}{\sqrt{l_n}}>0$ and that $|\te_\ell|\leq M^{\om(\ell)}\leq M^{c\frac{\ln(n)}{\ln(\ln(n))}}$
we see that for all $n$ large enough and $|t|$ small enough  we have
$$
\frac{|t|^3\,\cdot\,|\te_\ell|^3}{6a_n^3}\leq 
\frac14 \frac{t^2|\te_\ell|^2}{a_n^2}.
$$
Therefore, for $t$ small enough and $n$ large enough for every $\ell\leq n$ we have 
\begin{equation}\label{Cosss}
|\cos(t\te_\ell/a_n)|\leq 1-\frac{(\te_\ell)^2t^2}{4a_n^2}\leq e^{-\frac{(\te_\ell)^2t^2}{4a_n^2}}    
\end{equation}
where the last inequality uses that $1-x\leq e^{-x}$ for all $x\geq0$. Hence,
$$
\left|\bbE\left[\prod_{\ell=1}^n\cos(t\te_\ell)\right]\right|\leq \bbE\left[\exp\left(-\frac{t^2}{4a_n^2}\sum_{\ell=1}^n\te_\ell^2\right)\right].
$$
Taking into account our assumption that $a_n^2=o(\|S_n\|_{L^2}^2)$ we conclude that,
$$
Z_n:=a_n^{-2}\sum_{\ell=1}^n\te_\ell^2\to \infty\text{ in }L^1
$$
In particular  for $t\not=0$ since $a_n^{-2}\sum_{\ell=1}^n\te_\ell^2\geq 0$, possibly along a subsequence,
$$
-\frac{t^2}{4a_n^2}\sum_{\ell=1}^n\te_\ell^2\to -\infty\text{ a.s.}
$$
Indeed, by taking a subsequence $(n_k)$ such that $\bbE\left[Z_{n_k}\right]\geq k^4$ and using the Markov inequality we see that 
$$
\bbP(Z_{n_k}\geq\sqrt{\bbE[Z_{n_k}]})\leq k^{-2}
$$
and so by the Borel Cantelli lemma $Z_{n_k}\to\infty$ almost surely.
Consequently, by the dominated convergence theorem for every nonzero $t$,
$$
\limsup_{n\to\infty}\left|\bbE[\cos(tS_n/a_n)]\right|\leq \lim_{n\to\infty}\bbE\left[e^{-\frac{t^2}{4a_n^2}\sum_{\ell=1}^n\te_\ell^2}\right]=0.
$$

What remains in order to complete the proof of the proposition is to show that for all $t\in\bbR\setminus\{0\}$ small enough we have
\begin{equation}\label{Showw}
\lim_{n\to\infty}\sum_{k=1}^{[n/2]}(-1)^k\sum_{A\subset N_n: |A|=2k}\bbE\left[\prod_{\ell\in A}\sin(t\te _\ell)\prod_{\ell\in N_n\setminus A}\cos(t\te _\ell)\right]=0. 
\end{equation}
Next, set $c_n=\left[n^{1/2+\varepsilon}\sqrt{\ln n}\right]$ for some $0<\varepsilon<1/4$.
We  claim that 
\begin{equation}\label{Claim}
\lim_{n\to\infty}\sum_{k=c_n}^{[n/2]}(-1)^k\sum_{A\subset N_n: |A|=2k}\bbE\left[\prod_{\ell\in A}\sin(t\te _\ell)\prod_{\ell\in N_n\setminus A}\cos(t\te _\ell)\right]=0. 
\end{equation}
In order to prove \eqref{Claim}, we note that for all $n$ large enough, 
\begin{equation}\label{sin}
|\sin(t\te_\ell/a_n)|\leq |t\te_\ell/a_n|\leq |t/a_n|M^{\om(\ell)}\leq |t/a_n|M^{c\frac{\ln(n)}{\ln(\ln(n))}}:=|t/a_n|M_n.    
\end{equation}
Note also that 
$$
\binom{n}{2k}\leq \frac{n^{2k}}{(2k)!}.
$$
Therefore, 
$$
\sum_{k=c_n}^{[n/2]}\sum_{A\subset N_n: |A|=2k}\left|\bbE\left[\prod_{\ell\in A}\sin(t\te_\ell/a_n)\prod_{\ell\in N_n\setminus A}\cos(t\te _\ell)\right]\right|
\leq\sum_{k=c_n}^{[n/2]}\frac{|ntM_n/a_n|^{2k}}{(2k)!}.
$$
Next, using the Lagrange form of the Taylor remainders of the function $g(x)=e^x$ we see that for every $x\geq0$ and $m\geq0$ we have 
$$
\left|\sum_{k=m+1}^{\infty}\frac{x^k}{k!}\right|=\left|e^x-\sum_{k=0}^{m}\frac{x^k}{k!}\right|\leq\frac{e^x|x|^{m+1}}{(m+1)!}.
$$
Applying this with $x=|t|nM_n/a_n$ and $m=2c_n-1$  for all $n$ large enough 
we conclude that
$$
\sum_{k=c_n}^{[n/2]}\frac{|tnM_n/a_n|^{2k}}{(2k)!}\leq \frac{e^{n|t|M_n/a_n}(nM_n|t|/a_n)^{2c_n}}{(2c_n)!}.
$$
By combining the above estimate to complete the proof of \eqref{Claim} it is enough to show that 
$$
\lim_{n\to\infty}\frac{e^{n|t|M_n/a_n}(n|t|M_n/a_n)^{2c_n}}{(2c_n)!}=0.
$$
Let $\rho>0$. Then 
$$
\frac{e^{n|t|M_n/a_n}(n|t|M_n/a_n)^{2c_n}}{(2c_n)!}<\rho
$$
is equivalent to 
$$
|t|nM_n/a_n+2c_n\ln (|t|nM_n/a_n)\leq \ln\rho+\ln((2c_n)!).
$$
Next, note that by Striling's approximation for all $m\in\bbN$ have  $\ln((2m)!)=(2m)\ln(2m)-m+O(\ln m)$. Therefore, 
for all $n$ large enough 
$$
\ln((2c_n)!)\geq 2c_n\ln(2c_n)-2c_n+O(\ln(c_n))\geq |t|nM_n/a_n+2c_n\ln (|t|nM_n/a_n)+|\ln\rho|
$$
where the inequality holds since $n/a_n=O(\sqrt{n\ln(n)})$ and $M_n=o(n^\varepsilon)$ for every $\varepsilon>0$. This prove \eqref{Claim}.

Now, in order to complete the proof of the proposition it is enough to show that for all $t\in\bbR\setminus\{0\}$, 
\begin{equation}\label{Sh}
 \lim_{n\to\infty}\sum_{k=1}^{c_n-1}(-1)^k\sum_{A\subset N_n:\,|A|=2k}\bbE\left[\prod_{\ell\in A}\sin(t\te_\ell/a_n)\prod_{\ell\in N_n\setminus A}\cos(t\te_\ell/a_n)\right]=0.   
\end{equation}
Let us take some $A\subset N_n$ such that $|A|=2k$ for $1\leq k\leq c_n-1$. 
By expanding the functions $\cos(x)$ and $\sin(x)$ and denoting $A^c=N_n\setminus A$ and $\bbN_0=\bbN\cup\{0\}$,
\begin{equation}\label{R n}
R_{t,n}(A):=\bbE\left[\prod_{\ell\in A}\sin(t\te_\ell/a_n)\prod_{\ell\in N_n\setminus A}\cos(t\te_\ell/a_n)\right]=    
\end{equation}
$$
\sum_{(m_\ell)\in\bbN^A}\,\sum_{(s_d)\in \bbN_0^{A^c}}(-1)^{\sum_{\ell\in A}(m_\ell+1)+\sum_{d\in A^c}s_d}\frac{t^{\sum_{\ell\in A}(2m_\ell-1)+2\sum_{d\in A^c}s_d}}{\prod_{\ell\in A}(2m_\ell-1)!\prod_{d\in A^c}(2s_d)!}\bbE\left[\prod_{\ell\in A}\te_\ell^{2m_\ell-1}\prod_{d\in A^c}\te_{d}^{2s_d}\right].
$$
Next, we note that since $2s_d$ and $2(m_\ell-1)$ are even and $X_p$ are symmetric and independent,
$$
\bbE\left[\prod_{\ell\in A}\te_\ell^{2m_\ell-1}\prod_{s\in A^c}\te_{d}^{2s_d}\right]=0
$$
if and only if 
$$
\bbE\left[\prod_{\ell\in A}\te_\ell\right]=0.
$$
Indeed, both expectations vanish if and only if each prime $p$ that divides some $\ell\in A$ divides an even number of $\ell$'s in $A$.
Therefore,
$$
\left(\bbE\left[\prod_{\ell\in A}\te_\ell\right]=0\right)\Longrightarrow\left(R_{t,n}(A)=0\right).
$$

Let us take some $k\leq c_n-1$ and  $A\subset N_n=\{1,...,n\}$ such that $|A|=2k$. We call the set $A$ good if either $A$ is not contained in $\text{SF}$ (square free numbers) or there exists at least one prime number $p$ that divides an odd amount of members $\ell$ of $A$. Then if $A$ is good we have 
$$
\bbE\left[\prod_{\ell\in A}\te_\ell\right]=R_{t,n}(A)=0
$$
since $\prod_{\ell\in A}\te_\ell$ includes $X_p$ with some odd power, unless $A$ is not contained in $\text{SF}$, but in this case $\te_\ell=0$ for some $\ell$ and so the above expectation anyway vanishes. Note that for $k=1$ all subsets of $N_n$ are good. Henceforth we assume that $k>1$. 
When $k>1$ we call a set $A$ bad if it is not good. Namely, if $A\subset N_n\cap\text{SF}$ and every prime number $p$ divides an even number of members $\ell\in A$, that is $\prod_{a\in A}a$ is a perfect square.
Let us take a bad set $A\subset N_n\cap\text{SF}$ with $|A|=2k, \,1<k<c_n$. 

Let us take some $k\leq c_n-1$ and  $A\subset N_n=\{1,...,n\}$ such that $|A|=2k$. We call the set $A$ good if either $A$ is not contained in $\text{SF}$ (square free numbers) or there exists at least one prime number $p$ that divides an odd amount of members $\ell$ of $A$. Then if $A$ is good we have 
$$
\bbE\left[\prod_{\ell\in A}\te_\ell\right]=R_{t,n}(A)=0
$$
since $\prod_{\ell\in A}\te_\ell$ includes $X_p$ with some odd power, unless $A$ is not contained in $\text{SF}$, but in this case $\te_\ell=0$ for some $\ell$ and so the above expectation anyway vanishes. Note that for $k=1$ all subsets of $N_n$ are good. Henceforth we assume that $k>1$. 
When $k>1$ we call a set $A$ bad if it is not good. Namely, if $A\subset N_n\cap\text{SF}$ and every prime number $p$ divides an even number of members $\ell\in A$, namely $\prod_{a\in A}a$ is a perfect square.  Let us take a bad set $A\subset N_n\cap\text{SF}$ with $|A|=2k, \,1<k<c_n$. Next, if $\cB_{2k,n}$ denotes the number of bad sets of size $2k$, 
by Applying \cite[Theorem 1.2]{Har19} we see that for all $2\leq k\leq \frac{c\ln(n)}{\ln(\ln(n))}$
we have
\begin{equation}\label{A}
\cB_{2k,n}\leq C_k\frac{n^k (\ln(n))^{(k(2k-3))}}{(2k)!}   
\end{equation}
where $C_k=e^{-2k^2\ln k-2k^2\ln(\ln(k))+O(k^2)}$ and $c_1$ is a positive absolute constant.  Indeed, in the classical Rademacher setting $\cB_{2k,n}(2k)!$ appears in the expression for $\bbE[(S_n)^{2k}]$, where the $(2k)!$ factor comes from the number of ways to arrange the indexes since we are only interested in sets.

Next, taking into account Assumption \ref{Ass}, to complete the proof of the proposition it is enough to show that
\begin{equation}\label{Sh1}
 \lim_{n\to\infty}\sum_{k=2}^{\frac{c_1\ln(n)}{\ln(\ln(n))}}(-1)^k\sum_{A\subset N_n:\,|A|=2k}\bbE\left[\prod_{\ell\in A}\sin(t\te_\ell/a_n)\prod_{\ell\in N_n\setminus A}\cos(t\te_\ell/a_n)\right]=0.  
\end{equation}
To prove \eqref{Sh1}, by
 \eqref{Cosss} and \eqref{sin},
$$
|R_{t,n}(A)|\leq (M_n|t|/a_n)^{2k}\bbE[e^{-\frac{t^2}{4a_n^2}\sum_{\ell\in N_n\setminus A}\te_\ell^2}], M_n=\max(1,M)^{c\frac{\ln(n)}{\ln(\ln(n))}}.
$$
Notice that 
$$
\bbE[e^{-\frac{t^2}{4a_n^2}\sum_{\ell\in N_n\setminus A}\te_\ell^2}]\leq \bbE[e^{-\frac{t^2}{4a_n^2}\sum_{\ell=1}^n\te_\ell^2}]e^{2k\frac{t^2}{4a_n^2}M_n}.
$$
Moreover, when $k\leq c_n$ then $2kM_n/a_n^2\leq 2c_nM_n/a_n^2\leq n^{1/2+\varepsilon}\ln(n)M_n/a_n^2\to 0$ as $n\to\infty$ and so there is a constant $C>0$ such that 
$$
\bbE[e^{-\frac{t^2}{4a_n^2}\sum_{\ell\in N_n\setminus A}\te_\ell^2}]\leq C\bbE[e^{-\frac{t^2}{4a_n^2}\sum_{\ell=1}^n\te_\ell^2}].
$$
Therefore, by \eqref{A} and since $R_{t,n}(A)=0$ for all good sets $A$,
we see that
$$
\left|\sum_{k=2}^{\frac{\ln(n)}{\ln(\ln(n))}}(-1)^k\sum_{A\subset N_n: \,|A|=2k}R_{t,n}(A)\right|=\left|\sum_{k=2}^{\frac{\ln(n)}{\ln(\ln(n))}}\sum_{A\subset N_n: \,|A|=2k,\,A\,\text{ is bad}}\,R_{t,n}(A)\right|
$$
$$
\leq C\bbE[e^{-\frac{t^2}{4a_n^2}\sum_{\ell=1}^n\te_\ell^2}]\sum_{k=2}^{\frac{\ln(n)}{\ln(\ln(n))}}C_k(\ln (n))^{2k^2}\frac{\left(a_n^{-1}|t|M_n\sqrt n\right)^{2k}}{(2k)!}
$$
next, let us analyze the term $C_k(\ln (n))^{2k^2}$. Let $g_n:[2,c\ln(n))\to\bbR$ be given by $g_n(x)=\frac{(\ln(n))^{2x^2}}{x^{2x^2}}$. Then 
$$
g_n(x)=\left(\frac{\ln n}{x}\right)^{2x^2}.
$$
Let $G_n(x)=\ln(g_n(x))$. Then 
$$
G_n'(x)=2x(\ln(\ln(n))-\ln x-1).
$$
It follows that the functions $G_n$ and $g_n$ have a maximum at the point $x=e^{-1}\ln(n)$ and they are increasing on $[2,e^{-1}\ln(n)]$. Therefore, if $k\leq c\frac{\ln(n)}{\ln(\ln(n))}$ then 
$$
C_k(\ln (n))^{2k^2}\leq Cg_n(c\frac{\ln(n)}{\ln(\ln(n))})\leq C'\ln(\ln(n)).
$$
We thus conclude that 
$$
\sum_{k=2}^{\frac{\ln(n)}{\ln(\ln(n))}}C_k(\ln (n))^{2k^2}\frac{\left(a_n^{-1}|t|M_n\sqrt n\right)^{2k}}{(2k)!}\leq C'\ln(\ln(n))e^{|t|M_n\sqrt n/a_n}.
$$
Now \eqref{Sh1} follows our assumption that $\te_\ell^2\geq c_0^{\om(\ell)}>0$ (i.e. $|X_p|\geq c_0$) and  
$$
a_n=o\left(n^{1/2}c_0^{-2c\frac{\ln(n)}{\ln(\ln(n))}}(\ln(\ln(\ln(n)))^{-1}\right),\,\ \text{ and } 
 a_n=o\left(n^{1/2}M_n^{-1}c_0^{-c\frac{\ln(n)}{\ln(\ln(n))}}\right)
 $$
since the latter conditions ensure that the term $\bbE[e^{-t^2a_n^{-2}\sum_{\ell=1}^n\te_\ell^2}]$ dominates the term $C'\ln(\ln(n))e^{|t|M_n\sqrt n/a_n}$. 
\end{proof}

\subsection*{Completion of the proof of Theorem \ref{Main}}
For the sake of contradiction, let us assume that (possibly along a subsequence) $S_n/a_n$ converges in distribution to some random variable $Z$. 
Then for every real $t$,
$$
\lim_{n\to\infty}\bbE[e^{itS_n/a_n}]=\bbE[e^{itZ}].
$$
Note that due to the convergence of $S_n/a_n$, Lemma \ref{L1} ensures that $\liminf_{n\to\infty}\frac{a_n^2}{l_n}>0$.
Thus by Proposition \ref{L2}, for all $t\in(-\delta_0,\delta_0)\setminus\{0\}$ we have 
\begin{equation}\label{upp}
\lim_{n\to\infty}\bbE[\cos(tS_n/a_n)]=0.   
\end{equation}
Let us write
$$
\bbE[e^{itS_n/a_n}]=\bbE[\cos(tS_n/a_n)]+i\bbE[\sin(tS_n/a_n)].
$$
Next, by \eqref{upp} and since $S_n/a_n\to Z$, for every real $t$,
$$
\bbE[e^{itZ}]=i\lim_{n\to\infty}\bbE[\sin(tS_n/a_n)].
$$
Let us take some $\varepsilon>0$. Then by the continuity of the characteristic function of $Z$,  for every sufficiently small nonzero $t$ we have 
$$
|\bbE[e^{itZ}]-1|<\varepsilon.
$$
Thus for all $n$ large enough we have 
$$
1+|\bbE[\sin(tS_n/a_n)]|^2=|i\bbE[\sin(tS_n/a_n)]-1|^2<(1+\varepsilon)^2
$$
which implies that $|\bbE[\sin(tS_n/a_n)]|^2<2\varepsilon+\varepsilon^2$. Therefore,
$$
\lim_{n\to\infty}\bbE[\sin(tS_n/a_n)]=0
$$
and so by \eqref{UP} we get that for every nonzero $t$, 
$$
\bbE[e^{itZ}]=\lim_{n\to\infty}\bbE[e^{itS_n/a_n}]=0.
$$
However for all  $t$ small enough $|\bbE[e^{itZ}]|>\frac12$ since  characteristic functions are always continuous.
\qed
\subsection*{Proof of Theorem \ref{MomThm2}}
Taking $a_n=1$ in \eqref{R n} and using \eqref{Coss1} we see that for all $r\in\bbN$,
$$
\bbE[(S_n)^{2r}]=\sum_{k=1}^{[n/2]}(-1)^k\sum_{A\subset N_n:\,|A|=2k}\Gamma_{k,r}(A)
$$
where 
$$
\Gamma_{k,r}(A)=\sum_{\substack{(s_d)\in\bbN_0^{A^c}, (m_\ell)\in\bbN^A\\ \sum_{\ell\in A}(2m_\ell-1)+2\sum_{d\in A^c}s_d=2r}}
\frac{(2r)!}{\prod_{\ell\in A}(2m_\ell-1)!\prod_{d\in A^c}(2s_d)!}\bbE\left[\prod_{\ell\in A}\te_\ell^{2m_\ell-1}\prod_{d\in A^c}\te_{d}^{2s_d}\right].
$$
Notice that $\Gamma_{k,r}(A)=0$ if $k>r2$ since $k\leq \sum_{\ell\in A}(2m_\ell-1)\leq 2r$. Thus, for all $n$ large enough
\begin{equation}\label{MomEq} 
 \bbE[(S_n)^{2r}]=\sum_{k=1}^{2r}(-1)^k\sum_{A\subset N_n:\,|A|=2k}\Gamma_{k,r}(A)   
\end{equation}
Now, arguing like at the end of the proof of Proposition \ref{L2} for all $n$ large enough we see that $\Gamma_{k,r}(A)\not= 0$ only for $O((\ln(n))^{r(2r-3)} n^{r})$ sets $A$. Therefore, 
$$
\bbE[(S_n)^{2r}]\leq C_r n^r(\ln(n))^{2r}\left(\max(1,M)\right)^{\max_{\ell\leq n}\om(\ell)}\leq C_r n^r(\ln(n))^{2(2r-3)}\left(\max(1,M)\right)^{c\frac{\ln(n)}{\ln(\ln(n))}}.
$$
\qed

\begin{acknowledgement*}
I would like to thank Ofir Gorodetsky for many insightful comments, detailed explanations and references.
\end{acknowledgement*}

\end{document}